\documentclass[10pt]{article}
\usepackage{amsfonts}
\usepackage{graphicx}
\usepackage{amsmath}
\usepackage{amsthm}
\usepackage{amssymb}
\usepackage{amscd}
\usepackage{color}

\usepackage{graphicx}
\usepackage{microtype}
\usepackage{enumerate}
\usepackage{fullpage}
\usepackage{pinlabel}
\usepackage[top=1.1in,bottom=1.5in,left=1.5in,right=1.5in]{geometry}
\usepackage{titling}
\theoremstyle{definition}
\newtheorem{theorem}{Theorem}[section]

\newtheorem{corollary}[theorem]{Corollary}
\newtheorem{example}[theorem]{Example}
\newtheorem{lemma}[theorem]{Lemma}

\newtheorem{proposition}[theorem]{Proposition}
\newtheorem{remark}[theorem]{Remark}

\newtheorem{definition}[theorem]{Definition}

\theoremstyle{plain}

\newcommand{\R}{\mathbb{R}}
\newcommand{\Z}{\mathbb{Z}}
\newcommand{\N}{\mathbb{N}}

\newcommand{\del}{\partial}

\DeclareMathOperator{\id}{id}
\DeclareMathOperator{\supp}{supp}
\DeclareMathOperator{\fix}{fix}
\DeclareMathOperator{\PL}{PL}
\DeclareMathOperator{\PSL}{PSL}
\DeclareMathOperator{\GL}{GL}
\DeclareMathOperator{\Homeo}{Homeo}
\DeclareMathOperator{\Diff}{Diff}
\DeclareMathOperator{\interior}{int}

\title{$\PL(M)$ admits no Polish group topology}
\author{Kathryn Mann}
\date{}
\begin{document}

\maketitle

\vspace{-1cm}
\begin{abstract}  
We show that the group of piecewise linear homeomorphisms of any compact PL manifold does not admit a Polish group topology.  This uses \emph{a)} new results on the relationship between topologies on groups of homeomorphisms, their algebraic structure, and the topology of the underlying manifold, and \emph{b)} new results on the structure of certain subgroups of $\PL(M)$.    The proof also shows that the group of piecewise projective homeomorphisms of $S^1$ has no Polish topology.  
\end{abstract}

\section{Introduction}

Many transformation groups exhibit remarkable links between their algebraic structure and topology.  In the case of groups of homeomorphisms of manifolds, there is an additional rich interplay between algebraic structure and group topology and the topology of the underlying manifold.  For example, Kallman used this perspective to show that many ``big" groups of homeomorphisms, such as the full homeomorphism group or diffeomorphism group of a manifold, admit a unique Polish (separable and completely metrizable) group topology \cite{Kallman}.   Other instances of this algebraic--topological relationship can be seen in the main results of \cite{Filipkiewicz}, \cite{Mann ETDS}, \cite{Hurtado}, and \cite{Mann aut cont}.  

Here, we study the group $\PL(M)$ of piecewise linear homeomorphisms of a manifold.  Recall that an orientation preserving homeomorphism $f$ of the $n$-cube $I^n$ is \emph{piecewise linear} if there exists a subdivision of $I^n$ into finitely many linear simplices so that the restriction of $f$ to each simplex is an affine linear homeomorphism onto its image.   A manifold $M$ has a $PL$ structure if it is locally modeled on $\left(I^n, \PL(I^n) \right)$, in which case $\PL(M)$ is the automorphism group of this structure.   This group is interesting for many reasons, including its algebraic structure (c.f. \cite{BS}, \cite{CR}) and its relationship with the Thompson groups.    As for topology, there has been sigificant historical interest in how best to topologize $\PL(M)$.  For instance, in \cite[problems 39--40]{Lashof}, three different topologies are proposed by Milnor, Stasheff, and Wall, and the choice of most appropriate topology appears to be unresolved.   

We ask if the group $\PL(M)$ admits a Polish group topology.   This question is inherently interesting from the perspective of descriptive set theory (see e.g. \cite{Rosendal} and references therein) but also interesting from the perspective of transformation groups.  Much like the group of diffeomorphisms of a manifold, $\PL(M)$ is not complete in (and arguably, not best described by) the compact-open topology inherited from $\Homeo(M)$.  However $\Diff(M)$ does have a different topology -- the standard $C^\infty$ topology --  that makes it a Polish group, and we ask whether $\PL(M)$ might as well.  
  
In \cite{CK}, Cohen and Kallman showed that $\PL(I)$ and $\PL(S^1)$ have no Polish group structure.  However, their proof uses 1-dimensionality in an essential way.  Here we follow a different strategy, giving an independent proof of the following stronger result.

\begin{theorem} \label{main thm}
Let $M$ be a $\PL$ manifold.  Then $\PL(M)$ does not admit a Polish group topology.   
\end{theorem} 

\noindent In fact, we will show that $\PL(M)$ is \emph{very far} from being a Polish group.   In the case of $M = I$, this distinction is  easy to summarize:  while $\PL(I)$ contains no free subgroups (see \cite{BS} or Lemma \ref{BS lem} below), we will show that for many groups of homeomorphisms, including $\PL(I)$, the generic pair of elements with respect to \emph{any} Polish group topology generate a free subgroup.   
Precisely, we prove the following.  

\begin{theorem}  \label{generic prop}
Let $G \subset \Homeo(M)$ be a group satisfying a ``local perturbations property" (c.f. definition \ref{pert def} below).  If $\tau$ is any Polish group topology on $G$, then the generic pair $(f,g)$ in $G\times G$ with the product topology generate a free subgroup.  
\end{theorem} 

Loosely speaking, the \emph{local perturbations property} is a statement that there exist many homeomorphisms supported on small sets and close to the identity in $G$.  While $\PL(M)$ contains many homeomorphisms supported on small sets, the application of Theorem \ref{generic prop} is not completely straightforward, as these need not \emph{a priori} be close to the identity in any Polish topology on $\PL(M)$.  However, we show that this is indeed the case, by proving a very general result on topologies on groups of homeomorphisms (Theorem \ref{pert prop}) illustrating again the rich relationship between the topology of $M$ and topologies of groups of homeomorphisms of $M$.

The situation is more complicated for higher dimensional manifolds, since there are many examples of free subgroups in $\PL(M)$ as soon as $M$ has dimension at least 2.  (See Section \ref{free sec}.)   However, we are able produce a proof very much in the same spirit as the $\PL(I)$ case.  In essence, the proof consists of describing a natural subgroup of $\PL(M)$, showing that this subgroup necessarily inherits any Polish group structure from $\PL(M)$, and finally that the subgroup is both large enough to have the local perturbations property and small enough to contain no free subgroup.  

Our strategy also applies to other transformation groups, such as the group of \emph{piecewise projective} homeomorphisms of $S^1$ discussed in \cite{Monod}.   In Section \ref{cor sec} we show the following.  

\begin{corollary} \label{projective cor}
The group of piecewise projective homeomorphisms of $S^1$ admits no Polish topology.  
\end{corollary} 

One can interpret these results as ``explaining" why we have yet to settle on a choice of topology for $\PL(M)$.  It is suspected that several other transformation groups, such as the group of bi-Lipschitz homeomorphisms of a manifold, or diffeomorphisms of intermediate regularity on a smooth manifold, fail to admit a Polish group topology.  Progress is made in \cite{CK} for the one-dimensional case, it would be interesting to extend these results to higher dimensional manifolds as well.  
\bigskip 

\noindent \textbf{Acknowledgements.} 
The author thanks Michael Cohen, Jake Herndon, Alexander Kupers, and Christian Rosendal for their comments and interest in this project.

\section{Properties of transformation groups}  \label{properties sec}
\subsection{General results}
We begin by discussing general constraints on topologies on groups of homeomorphisms.  The broad idea is that any topology on a sufficiently rich group of homeomorphisms ``sees" to some extent the topology of $M$.  The first result along these lines is the following lemma of Kallman.  

\begin{lemma}[Kallman \cite{Kallman}]  \label{Kal lemma}
Let $M$ be a topological manifold, and let $G \subset \Homeo(M)$ have the property 
\begin{quote} $(\ast)$ for each nonempty open $U \subset M$, there exists a non-identity map $g_U \in G$ which fixes $M \setminus U$ pointwise. 
\end{quote}
If $\tau$ is any Hausdorff group topology on such a group $G$, then each set of the form
$C(U,V) := \{f \in G : f(\overline{U})\subseteq \overline{V} \}$ is closed in $(G, \tau)$. 
\end{lemma} 

\noindent A proof of Lemma \ref{Kal lemma} for the case of $I = M$ is given in \cite[Lemma 2.2]{CK}, but it applies equally well in the general case.  The outline is as follows: first, one uses property $(\ast)$ to show that $C(U, V)$ is the intersection of all sets of the form 
$$F_{U', W'} : =  \{f \in \Homeo(M) : f g_{U'} f^{-1} \text{ commutes with } g_{W'} \}$$ 
where $U'$ is open in $U$ and $W'$ is open in $M \setminus V$.  Now each $F_{U', W'}$ is closed, since it is the pre-image of the identity under the (continuous) commutator map $\Homeo(M) \to \Homeo(M)$ given by $f \mapsto  [f g_{U'} f^{-1}, g_{W'}]$.  
\\

We will work with the following strengthening of condition $(\ast)$.  Note that this condition is satisfied by $\PL(M)$, as well as many other familiar transformation groups such as $\Diff_0(M)$, the group of volume-preserving diffeomorphisms of a manifold, etc. 

\begin{definition}  \label{pert def}
Say that $G \subset \Homeo(M)$ satisfies the \emph{local perturbations property} if, for each open set $U \subset M$ and point $y \in U$, the set $\{f(y) : f|_{M\setminus U} = \id \}$ is uncountable.  
\end{definition}  

The next theorem explains our choice of the name ``local perturbations property", it says that such perturbations can be found in any neighborhood of the identity (i.e. \emph{locally} in the group--theoretic sense) in $G$.  

\begin{theorem}[Local perturbations exist] \label{pert prop}
Suppose that $G \subset \Homeo(M)$ has the local perturbations property and $\tau$ is a separable, metrizable topology on $G$.   Then for any open set $U \subset M$, point $y \in U$, and open neighborhood of the identity $N$ in $(G, \tau)$, there exists $f \in N$ such that $f(y) \neq y$ and $f|_{M \setminus U} = \id$.  
\end{theorem}

\begin{proof} 
It suffices to prove the Theorem in the case where $U$ is a small ball about $y$.  Given such $U \subset M$ and $y \in U$, let $H = \{f : f|_{M\setminus U} = \id\}$.   Since $(G, \tau)$ is separable and metrizable, the subset topology on $H$ is also separable.  

We now claim that, for any neighborhood $N$ of the identity in $G$, the neighborhood $N \cap H$ of the identity in $H$ contains some $f$ such that $f(y) \neq y$.  To see this, let $\{h_i : i \in \N \}$ be a countable dense subset of $H$. Then 
$$H = \bigcup \limits_{i=1}^\infty h_i(N \cap H).$$
If we had $h(y) = y$ for all $h \in N \cap H$, then the set of images $\{f(y) : f \in H\} = \{h_i(y)\}$ would be countable.  This contradicts the local perturbations property.  

\end{proof}

Theorem \ref{pert prop} readily generalizes to groups of homeomorphisms fixing a submanifold or, in the case where $\del M \neq 0$, those fixing the boundary of $M$.  For simplicity, we state only the boundary case.  Let $\Homeo(M, \del)$ denote the group of homeomorphisms of $M$ that fix $\del M$ pointwise.  

\begin{proposition} \label{pert prop 2}
Suppose that $G \subset \Homeo(M, \del)$ has a separable metrizable topology $\tau$, and suppose that the condition in the \emph{local perturbations property} is satisfied for every point $y$ in the interior of $M$.  Then for any closed set $X \subset M$, interior point $y \notin X$, and open neighborhood of the identity $N$ in $(G, \tau)$, there exists $f \in N$ such that $f(x) = x$ for all $x \in X$, but $f(y) \neq y$.  
\end{proposition}

\noindent The proof is exactly the same as that of Proposition \ref{pert prop}.  

\subsection{Generic free subgroups}
Using Theorem \ref{pert prop}, we can now prove Theorem \ref{generic prop} on generic free subgroups.  Recall that this is the statement that, for any group $G \subset \Homeo(M)$ with the local perturbations property, and any Polish group topology $\tau$ on $G$, the $\tau$-generic pair of elements $(f, g) \in G\times G$ generate a free group.    Our proof will actually show that the generic pair generates a free group whose action on $M$ is such that some point has trivial stabilizer.

\begin{proof}[Proof of Theorem \ref{generic prop}]
Assume that $G$ has a Polish topology, so $G \times G$ is a Baire space.
For each non-trivial, reduced word $w \in F_2$, define a set 
$X_w := \{(f, g) \in G \times G : w(f,g)= \id\}$.  This is the pre-image of the singleton $\{id\}$ under the (continuous) map $G \times G \to G$ given by $(f, g) \mapsto w(f,g)$, so is closed.  We will show that $X_w$ has empty interior, in which case the generic pair $(f,g)$ does not lie in any $X_w$, and hence generates a free group.  

Assume for contradiction that some $X_w$ has nonempty interior and let $(f, g) \in \interior(X_w)$.   
Write $w(f,g) = t_k \ldots t_1$ as a reduced word, where each $t_i \in \{f^{\pm 1}, g^{\pm 1} \}$.  

Choose any point $y _0 \in M$, and let $y_i = t_i \ldots t_1(y_0)$.  Let $m$ be the minimum integer such that the points $y_0, y_1, y_2, \ldots, y_m$ are \emph{not} all distinct.  Since $w(f,g) = \id$, we have $y_k = y_0$ and so $1 \leq m \leq k$.  Let $U$ be a small neighborhood of $y_{m-1}$ disjoint from $\{y_0, y_1, \ldots, y_{m-2}\} \setminus \{y_m\}$, and such that $t_m(U)$ is also disjoint from $\{y_0, y_1, \ldots, y_{m-1}\} \setminus \{y_m\}$.  
By Theorem \ref{pert prop}, for any neighborhood $N$ of the identity in $G$, there exists $h \in N$ such that $h(y_{m-1}) \neq y_{m-1}$ and $h$ restricts to the identity on the complement of $U$.  
Modify $t_{m}$ (which is either $f$, $g$, $f^{-1}$, or $g^{-1}$) by replacing it with $t_{m}\circ h$, and leaving the other free generator unchanged.  This gives a new pair $(f', g')$ that still lies in the interior of $X_w$, provided $N$ was chosen small enough.

We claim that, after this modification, the images of $y$ under the first $m$ initial strings of $w(f', g')$ -- adapting the previous notation, these are the points $t_i' \ldots t_1'(y)$, for $0 \leq i \leq m$, where $t_i' \in  \{(f')^{\pm 1}, (g')^{\pm 1} \}$ -- are now all distinct.  In fact, we will have $t_i' \ldots t_1'(y) = t_i \ldots t_1(y)$ for each $i < m-1$.  To see this, note that for each generator $t_i$, we have $t_i(y_{i-1}) = y_i$, except in the (intended) case $i=m$, or possibly if $t_{m-1} = t_m^{-1}$, in which case $t'_{m-1} = h^{-1}t_{m-1}$ and we would have $t'_{m-1}(y_{m-2}) =  h^{-1}(y_{m-1})$.  But this case is excluded by requiring that $w$ be a reduced word.
As $t'_m(y_{m-1}) \neq y_m$, and $t'_m(y_{m-1}) \in t_m(U)$, this shows that the points $t_i' \ldots t_1'(y)$ are all distinct.  

If $w(f', g') \neq \id$, we are already done.   Otherwise, we may repeat the procedure, again perturbing a generator in the neighborhood of the first repeated point in the sequence of images of $y$ under initial subwords of $w(f', g')$.   The process terminates when we arrive at some pair $(f^{(k)}, g^{(k)})$ in the interior of $X_w$ satisfying either $w(f^{(k)}, g^{(k)}) \neq \id$ or the more specific condition $w(f^{(k)}, g^{(k)})(y) \neq y$.  This contradicts the definition of $X_w$.  

\end{proof}

\begin{remark} Our proof borrowed notation (and some inspiration) from Ghys' proof that the generic pair of homeomorphisms of the circle, with respect to the standard $C^0$ topology, generate a free group \cite[Prop 4.5]{Ghys Ens}.    
The difference here is that we know much less about the topology on $G$ than we do about the $C^0$ topology on $\Homeo(S^1)$.  In particular, we don't even know whether the evaluation maps $G \times M \to M$ given by $(g, x) \mapsto g(x)$ are continuous -- a fact used directly in \cite{Ghys Ens}.  
\end{remark}


\begin{remark}[Relative case of Theorem \ref{generic prop}] \label{relative rk}
Using Proposition \ref{pert prop 2} in place of Theorem \ref{pert prop}, the proof above shows that whenever  $G \subset \Homeo(M, \del)$ has a Polish group topology and satisfies the local perturbations property, then the generic pair of elements of $G$ generate a free group.  
\end{remark}

\section{Free groups in $\PL(I^n)$}  \label{free sec}

We use the following result of Brin and Squier.  

\begin{lemma}[\cite{BS}] \label{BS lem}
$\PL(I)$ contains no free subgroup.  More specifically, the subgroup generated by any two elements $f, g \in \PL(I)$ is either abelian or contains a copy of $\Z^\infty$.  
\end{lemma}

The proof is not hard, for completeness we give a sketch here.   Recall the standard notation $\supp(f)$ for the \emph{support of f}, the closure of the set $\{x \in M : f(x) \neq x\}$.  

\begin{proof}
Let $f, g \in \PL(I)$. If a point $x$ is fixed by both $f$ and $g$, then the derivative of the commutator $[f,g]$ at $x$ is 1, and it follows that $[f,g]$ is the identity on a neighborhood of $x$.  This shows that $\supp([f,g])$ is contained in $I \setminus (\fix(f) \cap \fix(g))$.  
Assuming that the subgroup generated by $f$ and $g$ is not abelian, let $W$ then be the (nonempty) subgroup consisting of homeomorphisms $w$ such that $\supp(w)$ is nonempty and contained in $I \setminus (\fix(f) \cap \fix(g))$.    

Choose some $w \in W$ such that $\supp(w)$ meets a \emph{minimum} number of connected components of $I \setminus (\fix(f) \cap \fix(g))$.   Let $A$ be a connected component of $I \setminus (\fix(f) \cap \fix(g))$ that meets $\supp(w)$, this is a closed interval.  Let $a$ and $b$ denote $\min \{\supp(w) \cap A\}$ and $\max \{\supp(w) \cap A\}$ respectively. 

As $\sup \{h(a) : h \in \langle f, g \rangle \}$ is fixed by $\langle f, g \rangle$, there exists some $h$ in $\langle f, g \rangle$ with $h(a) > b$.  It follows that $h w h^{-1}$ and $w$ have disjoint support on $A$, so $[hwh^{-1}, w]$ restricts to the identity here.  Since $\supp(w)$ was assumed to meet a minimum number of connected components of $I \setminus (\fix(f) \cap \fix(g))$, we must have $\supp([hwh^{-1}, w]) = \emptyset$, i.e. $hwh^{-1}$ and $w$ commute.   This process can be repeated iteratively, finding $h_n$ that displaces the support of $w$ off of $\bigcup_{i< n} \supp(h_i w h_i ^{-1}) \cap A$, giving a copy of $\Z^\infty$ in $\langle f, g\rangle$.  
\end{proof} 

By contrast, as soon as $\dim(M) = n \geq 2$, the groups $\PL(M)$, and $\PL(M, \del)$ contain many free subgroups.   A number of examples are given in \cite{CR}, the easiest ones are the following. 

\begin{example}
Let $n \geq 2$, and consider a free subgroup of $\GL_n(\R)$ freely generated by $\alpha$ and $\beta$.   For any $p \in M$, there exist $\PL$ homeomorphisms $f$ and $g$, fixing $p$, supported on a neighborhood of $p$, and with derivatives $Df(p) = \alpha$ and $Dg(p)=\beta$.  Taking the derivative at $p$ defines an injective homomorphism from the group generated by $f$ and $g$ to a free subgroup of $\GL_n(\R)$, hence $f$ and $g$ satisfy no relation.  
\end{example}

Our next goal is to show that we can exclude these free subgroups by restricting our attention to homeomorphisms that preserve a $1$-dimensional foliation.  

\begin{definition}
Let $\PL(I^n, \mathcal{F})$ denote the subgroup of $\PL(I^n)$ consisting of homeomorphisms that preserve each leaf of the foliation of $I^n$ by vertical lines $\{x\} \times I$, where $x \in I^{n-1}$.  
\end{definition}

\begin{proposition}  \label{no free prop} 
$\PL(I^n, \mathcal{F})$ does not contain a free subgroup.  
\end{proposition} 

\begin{proof}
Let $f, g \in \PL(I^n, \mathcal{F})$.  The restrictions of $f$ and $g$ to any vertical line $L = \{x\} \times I$ are piecewise linear homeomorphisms of $L$.  

By Lemma \ref{BS lem}, the restriction of $f$ and $g$ to $L$ generate a group that is either abelian or contains a copy of $\Z^\infty$, in particular, there are nontrivial words $u$ and $v$ in $\langle f, g \rangle$ such that $u$ and $v$ restrict to the identity on $L$ and such that $[u,v]$ is not the trivial word.   
Notice that if $u \in \PL(I^n, \mathcal{F})$ restricts to the identity on $L$, then after identifying $L$ with the $n^{th}$ coordinate axis, $u$ is locally linear of the form 
$$\begin{pmatrix}
1 & 0 & \ldots & 0 & a_1 \\
0 & 1 & \ldots & 0 & a_2 \\
\vdots \\
0 & 0 & \ldots & 1 & a_{n-1}\\
0 & 0 & \ldots & 0 & 1 
\end{pmatrix}
$$
and the collection of all such linear maps forms an abelian group.   In particular, the commutator $[u,v]$ agrees with the identity on a \emph{neighborhood} of $L$.  
Shrinking this neighborhood if needed, we may take it to be of the form $U \times I$, where $U$ is a neighborhood of $x$ in $I^{n-1}$.

Now consider the collection of all open sets of the form $U' \times I$, such that
\begin{itemize}
\item $U'$ is open in $I^{n-1}$ and 
\item There exists some nontrivial reduced word $w$ in $f$ and $g$ with $w|_{U' \times I} = \id$.  
\end{itemize} 
The argument given above shows that this collection of sets forms an open cover of $I^n$.  Let $\{ U_1 \times I$, \ldots, $U_m \times I \}$ be a finite subcover of \emph{minimal} cardinality, and for each $1 \leq i \leq m$ let $w_i$ be a nontrivial word that restricts to the identity on $U_i \times I$.  We claim that $m = 1$, and therefore $f$ and $g$ satisfy a nontrivial relation.  

To see that $m=1$, assume for contradiction that we have more than one set in the cover and choose $i$ and $j$ such that $U_i \cap U_j \neq \emptyset$.  If $[w_i, w_j]$ reduces to the trivial word, then $w_i$ and $w_j$ would both be powers of some word $w'$.  Since $\PL(I)$ is torsion-free, this implies that $w'$ restricts to the identity on both $U_i$ and $U_j$, so we could replace our cover with a smaller one, using the single set $(U_i \cup U_j) \times I$ on which $w'$ is identity.   Thus, the assumption of minimal cardinality of the cover implies that $[w_i, w_j]$ is a nontrivial word.  However, since $w_i$ pointwise fixes $U_i \times I$ and $w_j$ pointwise fixes $U_j \times I$, the commutator $[w_i, w_j]$ restricts to the identity on $(U_i \cup U_j) \times I$, and this again contradicts our choice of a minimal cover.  

\end{proof}

\section{Completing the proof}

We now put together our previous work to finish the proof of Theorem \ref{main thm}, starting with the special case of $M = I^n$.  Suppose for contradiction that $\PL(I^n)$ admits a Polish topology.   Let $\PL(I^n, \mathcal{F})$ be the subgroup of vertical line preserving homeomorphisms defined in the previous section.   We claim that $\PL(I^n, \mathcal{F})$ is a closed subgroup, and hence Polish.  This is a consequence of the following general lemma.  

\begin{lemma} \label{product lem}
Let $M = A \times B$ be a product manifold.  
Let $G \subset \Homeo(M)$ be a subgroup satisfying condition $(\ast)$, and $\tau$ a Hausdorff group topology on $G$.  Then 
$$G(B) := \{f \in G : f(\{a\} \times B) = \{a\} \times B \text{ for all } a \in A \}$$ 
is a closed subgroup.  
\end{lemma}

\begin{proof}
We show that $G(B)$ is an intersection of sets of the form $C(U, V)$, hence is closed by Lemma \ref{Kal lemma}.  
Consider the collection $\Lambda$ of sets of the form $\{U'\} \times B$, where $U'$ is open in $A$.  Then 
$$G(B) = \bigcap_{U \in \Lambda} C(U, U).$$ 
Indeed, if $f \in \Homeo(M)$ satisfies $f(\{U'\} \times B) \subset \{U'\} \times B$ for each $U'$ in a neighborhood basis of $a \in A$, then $f(\{a\} \times B) = \{a\} \times B$.  This gives the inclusion of $\bigcap_{U \in \Lambda} C(U, U)$ in $G(B)$, and the reverse inclusion is immediate.  
\end{proof}

To continue the proof of the Theorem, note that $\PL(I^n, \mathcal{F})$ also satisfies the local perturbations property -- for example, given $y \in M$ and any neighborhood $U$ of $y$, one can define for each $t \in (0, \epsilon)$ an element of $\PL(I^n, \mathcal{F})$ supported on $U$, and agreeing with $(x_1, ..., x_{n-1}, x_n) \mapsto (x_1, ... , x_{n-1}, x_n +t)$ on a small linear simplex containing $y$.  
Thus, by Theorem  \ref{generic prop} the generic $f, g \in \PL(I^n, \mathcal{F})$ generate a free subgroup.  This contradicts Proposition \ref{no free prop}, so we conclude that $\PL(I^n)$ has no Polish topology.  
\bigskip

This strategy also works to show that the group $\PL(I^n, \del)$ of piecewise linear homeomorphisms of $I^n$ that pointwise fix the boundary admits no Polish group topology.  The proof of Lemma \ref{product lem} shows that the subgroup of homeomorphisms in $\PL(I^n, \del)$ that preserve each vertical line is closed, hence Polish.  It also satisfies the (relative) local perturbations property.  Now Remark \ref{relative rk} implies that the generic pair of elements generate a free group, which is again a contradiction.  

For the general case, let $M$ be an $n$-dimensional $\PL$ manifold, and assume again for contradiction that $\PL(M)$ has a Polish group topology.  Let $A \subset M$ be a linearly embedded copy of $I^n$, and let $G \subset \PL(M)$ be the subgroup of homeomorphisms that restrict to the identity on $M \setminus A$.   We claim that  
$$G = \bigcap \limits_{U' \subset{M \setminus \bar{A}} \text{ open}} C(U', U'),$$
and hence $G$ is a closed subgroup.  That $G \subset C(U', U')$ for any $U' \subset M \setminus \bar{A}$ is immediate, to see the reverse inclusion, take any point $x \in M \setminus \bar{A}$. If $f(U') \subset U'$ for all sets $U'$ in a neighborhood basis of $x$, then $f(x) = x$.  

Since $G$ is closed, it is also a Polish group.  As $G \cong \PL(I^n, \del)$, this contradicts the case proved above, and completes the proof of Theorem \ref{main thm}. 

 \qed

\subsection{Piecewise projective homeomorphisms}  \label{cor sec}

A homeomorphism $f$ of $S^1$ is \emph{piecewise projective} if there is a partition of $S^1$ into finitely many intervals such that the restriction of $f$ to each interval agrees with the standard action of $\PSL(2,\R)$ by projective transformations on $\R P ^1 = S^1$.  
Much like $\PL(M)$, this group has a rich algebraic structure: among its subgroups are counterexamples to the Von Neumann conjecture (see \cite{Monod}), and the full group is closely related to a proposed ``Lie algebra" for the group $\Homeo(S^1)$ given in \cite{MP}.  

We now prove Corollary \ref{projective cor}, the analog of Theorem \ref{main thm} for piecewise projective homeomorphisms. 
Let $G$ denote the group of all piecewise projective homeomorphisms of $S^1$, let $I \subset S^1$ be a small interval, and let $H \subset G$ be the subgroup of homeomorphisms pointwise fixing $I$.  Suppose that $G$ is given a Polish group topology.   Since we have 
$$H = \bigcap \limits_{U \text{ open},\, U \subset I} C(U, U)$$
$H$ is a closed subgroup, hence Polish.  

Note also that $H$ has the local perturbations property (in the modified sense for groups of homeomorphisms fixing a submanifold), so it follows from Theorem \ref{generic prop} that the generic pair of elements of $H$ generate a free group.  However, the same argument as in Lemma \ref{BS lem} shows that the subgroup generated by any two elements of $H$ is either metabelian or contains a copy of $\Z^\infty$; in particular, it is not free (this is also proved in Theorem 14 of \cite{Monod}).    This gives a contradiction, showing that $G$ cannot have a Polish group topology.  

\qed



\vspace{.3in}

Dept. of Mathematics 

University of California, Berkeley  

970 Evans Hall

Berkeley, CA 94720 

E-mail: kpmann@math.berkeley.edu

\end{document}